\documentclass[preprint,12pt,3p]{elsarticle}

\usepackage{amsmath, amssymb, amsfonts, amsbsy, amsthm, latexsym, epsfig, txfonts, palatino, pifont, color, tikz, graphicx, framed}
\usetikzlibrary {positioning}
\definecolor {processblue}{cmyk}{0.96,0,0,0}

\theoremstyle{definition}
\newtheorem{theorem}{Theorem}

\newtheorem{rk}{Remark}
\newtheorem{df}{Definition}

\newcommand{\field}[1]{\mathbb{#1}}
\newcommand{\G}{\field{G}}

\newcommand{\C}{\field{C}}
\newcommand{\W}{\field{W}}

\makeatletter
\g@addto@macro\@floatboxreset\centering
\makeatother

\begin{document}

\begin{frontmatter}

\title{On Attractors of Isospectral Compressions of Networks}

\author{Leonid Bunimovich}
\ead{bunimovh@math.gatech.edu}

\author{Longmei Shu\corref{cor1}}
\address{School of Mathematics,
Georgia Institute of Technology,
Atlanta, GA 30332-0160 USA}

\cortext[cor1]{Corresponding author}

\ead{lshu6@math.gatech.edu}

\begin{abstract}
In the recently developed theory of isospectral transformations of networks isospectral compressions are performed with respect to some chosen characteristic (attribute) of nodes (or edges) of networks. Each isospectral compression (when a certain characteristic is fixed) defines a dynamical system on the space of all networks. It is shown that any orbit of such dynamical system which starts at any finite network (as the initial point of this orbit) converges to an attractor. Such attractor is a smaller network where a chosen characteristic has the same value for all nodes (or edges). We demonstrate that isospectral contractions of one and the same network defined by different characteristics of  nodes (or edges) may converge to the same as well as to different attractors. It is also shown that spectrally equivalent with respect to some characteristic networks could be non-spectrally equivalent for another characteristic of nodes (edges). These results suggest a new constructive approach to analysis of networks’ structures and to comparison of topologies of different networks.
\end{abstract}

\begin{mscc}
05C50 \sep 15A18
\end{mscc}


\end{frontmatter}

\section{Introduction}
Arguably the major scientific buzzword of our time is a "Big Data". When talking about Big Data people usually refer to (huge) natural networks in communications, bioinformatics, social sciences, etc, etc, etc. In all cases the first idea and hope is to somehow reduce these enormously large networks to some smaller objects while keeping, as much as possible, information about the original huge network. 

In practice almost all the information about real world networks is contained in their adjacency matrices \cite{newman06,newman10}. An adjacency matrix of a network with $N$ elements is the $N\times N$ matrix with zero or one elements. An $(i,j)$ element equals one if there is direct interaction between the elements number {i} and number {j} of a network. In the graph representation of a network it corresponds to the existence of an edge (arrow) connecting a node $i$ to the node $j$. Otherwise an $(i,j)$ element of the adjacency matrix of a network equals zero. it is very rarely \cite{,newman06,newman10} that the strength of interaction of the element (node) $i$ with the element (node) $j$ is also known. In such cases a network is represented by a weighted adjacency matrix where to each edge $(i,j)$ corresponds (instead of one) a weight equal to the strength of interaction along this edge. 

Therefore a problem of compression of information about a network is essentially a problem of compression of the weighted adjacency matrix of a network. It is a basic fact of linear algebra that all the information about a matrix is contained in its spectrum (collection of all eigenvalues of a matrix) and in its eigenvectors and generalized eigenvectors.

Recently a constructive rigorous mathematical theory was developed which allows to compress (reduce) matrices and networks while keeping ALL the information about their spectrum and eigenvalues. This approach was successfully applied to various theoretical and applied problems \cite{webb14}. The corresponding transformations of networks were called Isospectral Transformations. This approach is not reduced just to compression of networks. It also allows e.g. to grow (enlarge) networks while keeping stability of their evolution (dynamics), etc (see \cite{webb14,webb17}).

In the present paper we develop further this approach by demonstration that isospectral compressions generate a dynamical system in the space of all networks. We prove that any orbit of such dynamical system converges to an attractor which is a smaller network than the network which was an initial point (network) of this orbit.
Consider any characteristic of nodes (or edges) of a network. Then collection of all nodes (or edges) with a fixed value of this characteristic defines a subset of the set of all nodes (edges).   
It is important to mention that the current graph theory is lacking classification of all graphs which have the same characteristic of the all nodes even for such basic and simplest characteristics as inner and outer degrees. Clearly any full graph where any two nodes are connected by an edge (in case of undirected graphs) or by two opposite edges (in case of directed graphs) has the same value of any characteristic at any node. Therefore all full graphs are attractors of isospectral contractions. However, there are other attractors as well for any characteristic of nodes, and there is no general classification/description of corresponding attractors. However one can find such attractors when dealing with a concrete network. Therefore it is a natural tool for analysis of real world networks.
We demonstrate that by choosing different characteristics of nodes or edges of networks one gets generally different attractors. Structure of such attractors give a new important information about a given network.

We also discuss the notions of weak and strong spectral equivalences of networks and show that classes of equivalence with respect to a weak spectral equivalence consists of a countable number of classes of strongly spectrally equivalent networks.
Our results could be readily applicable to analysis of any (directed or undirected, weighted or unweighted) networks.

\section{Isospectral Graph Reductions and Spectral Equivalence}
In this section we recall definitions of the isospectral transformations of graphs and networks.

Let $\W$ be the set of rational functions of the form $w(\lambda)=p(\lambda)/q(\lambda)$, where $p(\lambda),q(\lambda)\in\C[\lambda]$ are polynomials having no common linear factors, i.e., no common roots, and where $q(\lambda)$ is not identically zero. $\W$ is a field under addition and multiplication \cite{webb14}.

Let $\G$ be the class of all weighted directed graphs with edge weights in $\W$. More precisely, a graph $G\in\G$ is an ordered triple $G=(V,E,w)$ where $V=\{1,2,\dots,n\}$ is the \emph{vertex set}, $E\subset V\times V$ is the set of \emph{directed edges}, and $w:E\to\W$ is the \emph{weight function}. Denote by $M_G=(w(i,j))_{i,j\in V}$ the \emph{weighted adjacency matrix} of $G$, with the convention that $w(i,j)=0$ whenever $(i,j)\not\in E$. We will alternatively refer to graphs as networks because weighted adjacency matrices define all static (i.e. non evolving) real world networks.

Observe that the entries of $M_G$ are rational functions. Let's write $M_G(\lambda)$ instead of $M_G$ here to emphasize the role of $\lambda$ as a variable. For $M_G(\lambda)\in\W^{n\times n}$, we define the spectrum, or multiset of eigenvalues to be $$\sigma(M_G(\lambda))=\{\lambda_i\in\C,i=1,\dots,n:\det(M_G(\lambda)-\lambda I)=0\}.$$

A path $\gamma=(i_0,\dots,i_p)$ in the graph $G=(V,E,w)$ is an ordered sequence of distinct vertices $i_0,\dots,i_p\in V$ such that $(i_l,i_{l+1})\in E$ for $0\le l\le p-1$. The vertices $i_1,\dots,i_{p-1}\in V$ of $\gamma$ are called \emph{interior vertices}. If $i_0=i_p$ then $\gamma$ is a \emph{cycle}. A cycle is called a \emph{loop} if $p=1$ and $i_0=i_1$. The length of a path $\gamma=(i_0,\dots,i_p)$ is the integer $p$. Note that there are no paths of length 0 and that every edge $(i,j)\in E$ is a path of length 1.

If $S\subset V$ is a subset of all the vertices, we will write $\overline S=V\setminus S$ and denote by $|S|$ the cardinality of the set $S$.

\begin{df}
\emph{(structural set).} Let $G=(V,E,w)\in\G$. A nonempty vertex set $S\subset V$ is a structural set of $G$ if 
\begin{itemize}
\item each cycle of $G$, that is not a loop, contains a vertex in $S$;
\item $w(i,i)\neq\lambda$ for each $i\in\overline S$.
\end{itemize}

In particular, if a structural set $S$ also satisfies $w(i,i)\neq\lambda_0,\forall i\in\overline S$ for some $\lambda_0\in\C$, then $S$ is called a $\lambda_0-$structural set.
\end{df}

\begin{rk}
It is easy to see that the complement of any single node is a structural set. Therefore for any subset $A$ of nodes of a network $G$, it is always possible to isospectrally compress the network $G$ to a network whose nodes belong to $A$ by removing the nodes in the complement of $A$ one after another.    
\end{rk}

\begin{df}
Given a structural set $S$, a \emph{branch} of $(G,S)$ is a path $\beta=(i_0,i_1,\dots,i_{p-1},i_p)$ such that  $i_0,i_p\in V$ and all $i_1,\dots,i_{p-1}\in\overline S$.
\end{df}
We denote by $\mathcal{B}=\mathcal{B}_{G,S}$ the set of all branches of $(G,S)$.
Given vertices $i,j\in V$, we denote by $\mathcal{B}_{i,j}$ the set of all branches in $\mathcal{B}$ that start in $i$ and end in $j$. For each branch $\beta=(i_0,i_1,\dots,i_{p-1},i_p)$ we define the \emph{weight} of $\beta$ as follows:
\begin{equation}
w(\beta,\lambda):=w(i_0,i_1)\prod_{l=1}^{p-1}\frac{w(i_l,i_{l+1})}{\lambda-w(i_l,i_l)}.
\end{equation}

Given $i,j\in V$ set
\begin{equation}
R_{i,j}(G,S,\lambda):=\sum_{\beta\in\mathcal{B}_{i,j}}w(\beta,\lambda).
\end{equation}

\begin{df}
\emph{(Isospectral reduction).} Given $G\in\G$ and a structural set $S$, the reduced adjacency matrix  $R_S(G,\lambda)$ is the $|S|\times |S|-$matrix with the entries $R_{i,j}(G,S,\lambda),i,j\in S$. This adjacency matrix $R_S(G,\lambda)$ on $S$ defines the reduced graph which is the isospectral reduction of the original graph $G$.
\end{df}

Now we recall the notion of spectral equivalence of networks (graphs).

Let $\W_\pi\subset\W$ be the set of rational functions $p(\lambda)/q(\lambda)$ such that $\text{deg}(p)\le\text{deg}(q)$, where $\text{deg}(p)$ is the degree of the polynomial $p(\lambda)$. And let $\G_\pi\subset\G$ be the set of graphs $G=(V,E,w)$ such that $w:E\to\W_\pi$. Every graph in $\G_\pi$ can be isospectrally reduced \cite{webb14}.

Two weighted directed graphs $G_1=(V_1,E_1,w_1)$ and $G_2=(V_2,E_2,w_2)$ are \emph{isomorphic} if there is a bijection $b:V_1\to V_2$ such that there is an edge $e_{ij}$ in $G_1$ from $v_i$ to $v_j$ if and only if there is an edge $\tilde e_{ij}$ between $b(v_i)$ and $b(v_j)$ in $G_2$ with $w_2(\tilde e_{ij})=w_1(e_{ij})$. If the map $b$ exists, it is called an \emph{isomorphism}, and we write $G_1\simeq G_2$.

An isomorphism is essentially a relabeling of the vertices of a graph. Therefore, if two graphs are isomorphic, then their spectra are identical. The relation of being isomorphic is reflexive, symmetric, and transitive; in other words, it's an equivalence relation.

In \cite{webb14} was introduced a notion of spectral equivalence of graphs (networks). It says that two networks $G$ and $H$ are spectrally equivalent if they reduce to isomorphic graphs in one step under the same rule for subset selection. Then in \cite{shu18} was introduced less restrictive notion of generalized spectral equivalence of graphs (networks). Namely, two networks are weakly spectrally equivalent if they reduce to isomorphic graphs in finite steps (not necessarily the same number of steps) under the same rule for subset selection.

A proof of the following theorem can be found in \cite{shu18}.

\begin{theorem}[Generalized Spectral Equivalence of Graphs]
Suppose that for each graph $G=(V,E,w)$ in $\G_\pi$, $\tau$ is a rule that selects a unique nonempty subset $\tau(G)\subset V$. Let $R_\tau$  be the isospectral reduction of $G$ onto $\tau(G)$. Then $R_\tau$ induces an equivalence relation $\sim$ on the set $\G_\pi$, where $G\sim H$ if $R_\tau^m(G)\simeq R_\tau^k(H)$ for some $m,k\in\mathbb{N}$.
\end{theorem}

\begin{rk}
Observe that we do not require $\tau(G)$ to be a structural subset of $G$. However there is a unique isospectral reduction \cite{webb14} (possibly via a sequence of isospectral reductions to structural sets if $\tau(G)$ is not a structural subset of $G$) of $G$ onto $\tau(G)$.

The notion of generalized spectral equivalence of networks (graphs) is weaker than the one considered in \cite{webb14}, where it was required that $m=k=1$. Therefore the classes of weakly spectrally equivalent networks are larger than the classes of spectrally equivalent networks considered in \cite{webb14}. Namely each class of equivalence in the weak sense consists of a countable number of equivalence classes in the (strong) sense of \cite{webb14}. In what follows we will refer to the spectral equivalence in the form introduced in \cite {webb14} as to a strong spectral equivalence, and to the notion of spectral equivalence introduced in \cite{shu18} as to the weak spectral equivalence. Both of the strong and weak notions of spectral equivalence could be of use for analysis of real world networks many of which have a hierarchical structure \cite{ACM}, \cite{JKAM}. 
\end{rk}

\section{Attractors of Isospectral Reductions}

Isospectral reductions of networks (graphs) define a dynamical system on the space of all networks. This dynamical system arises by picking any node (edge) of a network and isospectrally reducing this network to a network where the set of nodes is a complement to a chosen node. The fact that such isospectral reductions form a dynamical system follows from the Commutativity theorem proved in \cite{webb14} which states that sequence of isospectral contractions over a set of nodes $A$ and then over the set of nodes $B$ gives the same result as isospectral reduction over $B$ followed by the one over $A$. Therefore to one and the same network (graph) $G$ correspond different orbits depending on the order in which we pick nodes of $G$ for reductions.

By doing that again and again it is possible to isospectrally reduce any network to a trivial network which has just one node, i.e. any node of $G$. It is clearly a senseless operation. However we can choose a reasonable rule which will help to recover and understand some intrinsic feature(s) of the structure (topology) of the network $G$.   
Generally a network can have many different structural sets. To make the isospectral contraction focused on specific properties of networks, we can add some specific rules to the selection of structural sets.

Before we do that, let us recall a few characteristics of nodes in a graph. (There are about ten-fifteen such characteristics of nodes and edges of networks which are all borrowed from the graph theory).

For a graph $G=(V,E,w)$, the indegree for a node $v\in V$, $d^-(v)$, is the number of edges that end in $v$. The outdegree $d^+(v)$ is the number of edges that start at $v$. Let's define $d(v)=d^-(v)+d^+(v)$ to be the sum of the indegree and outdegree for any node.

Let $\sigma_{st}$ be the total number of shortest paths from node $s$ to node $t$, and $\sigma_{st}(v)$ is the number of those paths that pass through $v$. Note that $\sigma_{st}(v)=0$ if $v\in\{s,t\}$ or if $v$ does not lie on any shortest path from $s$ to $t$. We call $$g(v)=\sum_{s\neq v}\sum_{t\neq v,s}\sigma_{st}(v)$$ the centrality/betweenness of node $v$.

\begin{theorem}
For any network and any characteristic of its nodes (edges), any orbit of a dynamical system generated by isospectral reductions with respect to the chosen characteristic, converges to an attractor which is such a network that all its nodes (edges) have one and the same value of this characteristic of nodes (edges). 
\end{theorem}
\begin{proof}
Each reduction removes at least one vertex (edge). Thus an orbit of a network under consecutive isospectral reductions becomes an attractor in no more than N steps, where $N:=|V|$ (or $N:=|E|$). Therefore an orbit of a finite network $G$ approaches an attractor in a finite number of steps which does not exceed the number of nodes (edges) in $G$. Such attractor always exists because any network can be isospectrally reduced to a graph with just one node. A process of consecutive isospectral reductions (i.e. an orbit of the corresponding dynamical system) will terminate at one node, if no one of the networks along this orbit had all its nodes (edges) with the same value of the characteristic that defined the isospectral reductions (i.e. the corresponding dynamical system).  Clearly in case of a "network" with only one node (edge) the values of all characteristics of all nodes (edges) are the same because there is only one node (edge).  
If $G$ is an infinite network then the corresponding orbit could be finite or infinite. 

\end{proof}

\begin{theorem}
The attractors of isospectral reductions with respect to different characteristics of one and the same network are generally different.
\end{theorem}
\begin{proof}
\begin {figure}[h]
\begin{tikzpicture}[-latex, auto, node distance = 0.5 cm and 0.5 cm, on grid, semithick, state/.style ={ circle, top color =white, bottom color = processblue!20, draw, processblue, text=blue, minimum width =0.1 cm}]
\node[state] (A) at (0, 0) {$1$};
\node[state] (B) at (-1,-1.3) {$2$};
\node[state] (C) at (0,-2.6) {$3$};
\node[state] (D) at (1.5,-2.6) {$4$};
\node[state] (E) at (2.5,-1.3) {$5$};
\node[state] (F) at (1.5,0) {$6$};
\node[state] (G) at (3.5,0) {$7$};
\node[state] (H) at (5,0) {$8$};
\node[state] (I) at (6,-1.3) {$9$};
\node[state] (J) at (5,-2.6) {$10$};
\node[state] (K) at (3.5,-2.6) {$11$};
\path[-] (A) edge (B) edge (C) edge (D) edge (F)
(B) edge (C) edge (D) edge (F)
(C) edge (D) edge (F)
(D) edge (E)
(E) edge (F) edge (G) edge (K)
(G) edge (J) edge (I) edge (H)
(H) edge (I) edge (J) edge (K)
(I) edge (J) edge (K)
(J) edge (K);
\end{tikzpicture}
\caption{A network which is an attractor with respect to degree but not with respect to centrality}\label{attractor}
\end{figure}
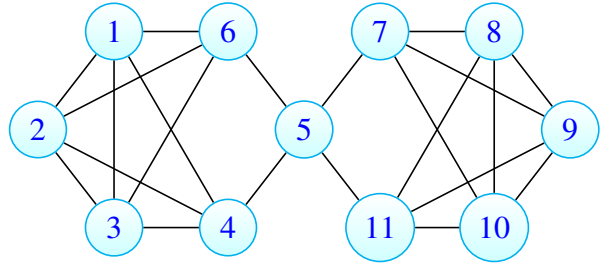

(i) In the example shown in the figure \ref{attractor}, all nodes have degree 4. This graph cannot be further reduced based on degree of nodes. However, the centrality of the nodes are different. If we count the number of shortest paths through each node, we can see $c(1)=c(2)=c(3)=c(8)=c(9)=c(10)=1,c(4)=c(6)=c(7)=c(11)=27,c(5)=66$. This graph can be further reduced based on centrality. Therefore for this network (graph)  attractors with respect to degree and to centrality are different.

(ii) The complete graph, where each and every node and edge have the same properties, can not be further reduced based on degree or other characteristics of a network. It is always an attractor. If we consider isospectral expansion (see \cite{webb17}) of a complete graph with respect to two different characteristics, then we get two different graphs (networks) with the same attractor with respect to these two characteristics. Clearly this attractor will be the initial complete graph.
\end{proof}

The result of theorem 3 is not surprising because different characteristics of nodes (or edges) define different dynamical systems on the space of all networks, and orbits of these different dynamical systems are also different. 

The following statement establishes that weakly as well as strongly spectrally equivalent networks have the same attractor if isospectral contractions are generated by the very same characteristic with respect to which these networks are spectrally equivalent. 

\begin{theorem}
Strongly as well as weakly spectrally equivalent with respect to some characteristic graphs have the same attractor under the dynamical system generated by isospectral compressions according to this characteristic.
\end{theorem}
\begin{proof}
Suppose graph $G$ is strongly spectrally equivalent to $H$ with respect to rule $\tau$, i.e. $R_\tau(G)\simeq R_\tau(H)=R$, and $G$ is weakly spectrally equivalent to $K$ w.r.t $\tau$, i.e. $R_\tau^l(G)\simeq R_\tau^m(K)=S$.

If $R$ is an attractor under $\tau$, then the attractor for $G$ as well as for $H$ is $R$. So $G$ and $H$ have the same attractor  $R$. Otherwise $G$ and $H$ have the same attractor, the attractor for $R$.

Similarly $G$ and $K$ have the same attractor.

Therefore the attractors for all three graphs, $G,H,K$ are the same under rule $\tau$. So all three networks (graphs) have the same attractor with respect to the rule $\tau$. 
\end{proof}

Very important fact is that networks could be spectrally equivalent with respect to one characteristic of nodes (edges) but not spectrally equivalent with respect to another characteristic. Therefore spectral equivalences built on different characteristics of nodes and edges allow to uncover various intrinsic (hidden) features of networks' topology.

We now present an example where networks are isomorphic for one characteristic but not for another.

Consider the graphs $G$ and $H$ in figure \ref{equivalence}.

\begin {figure}[h]
\begin {tikzpicture}[-latex ,auto ,node distance =2 cm and 2cm ,on grid ,
semithick ,
state/.style ={ circle ,top color =white , bottom color = processblue!20 ,
draw,processblue , text=blue , minimum width =0.5 cm}]
\node at (1,-3.75) {Graph $G$};
\node[state] (A) at (0, 0) {$1$};
\node[state] (B) at (2,0) {$2$};
\node[state] (C) at (1,-1.5) {$3$};
\node[state] (D) at (-1.5,0) {$4$};
\node[state] (E) at (3.5,0) {$5$};
\node[state] (F) at (1,-2.75) {$6$};
\path (A) edge [] node [above] {$1$} (B);
\path (A) edge [] node [left] {$1$} (C);
\path (B) edge [] node [right] {$1$} (C);
\path (A) edge [bend left = 15] node [below] {$1$} (D);
\path (D) edge [bend left = 15] node [above] {$1$} (A);
\path (B) edge [bend left = 15] node [above] {$1$} (E);
\path (E) edge [bend left = 15] node [below] {$1$} (B);
\path (C) edge [bend left = 15] node [right] {$1$} (F);
\path (F) edge [bend left = 15] node [left] {$1$} (C);
\path (A) edge[loop above] node [] {$1/\lambda$} (A);
\path (B) edge [loop above] node [] {$1/\lambda$} (B);
\path (C) edge [loop right] node [] {$1/\lambda$} (C);

\node at (7,-3.75) {Graph $H$};
\node[state] (G) at (6,0) {$1$};
\node[state] (H) at (8,0) {$2$};
\node[state] (I) at (7,-1.5) {$3$};
\node[state] (J) at (9.5,0) {$5$};
\node[state] (K) at (7,-2.75) {$6$};
\path (G) edge [] node [above] {$1$} (H)
(G) edge [] node [left] {$1$} (I)
(H) edge [] node [right] {$1$} (I)
(H) edge [bend left = 15] node [above] {$1$} (J)
(J) edge [bend left = 15] node [below] {$1$} (H)
(I) edge [bend left = 15] node [right] {$1$} (K)
(K) edge [bend left = 15] node [left] {$1$} (I)
(G) edge [loop above] node [] {$2/\lambda$} (G)
(H) edge [loop above] node [] {$1/\lambda$} (H)
(I) edge [loop right] node [] {$1/\lambda$} (I);
\end{tikzpicture}
\caption{Original networks}
\label{equivalence}
\end{figure}
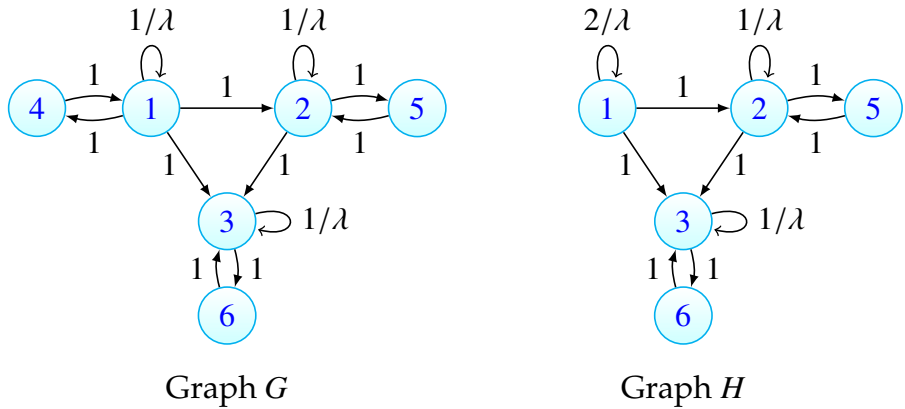

Their adjacency matrices are
$$M_G=\begin{pmatrix}1/\lambda & 1 & 1 & 1 & 0 & 0\\
0 & 1/\lambda & 1 & 0 & 1 & 0\\
0 & 0 & 1/\lambda & 0 & 0 & 1\\
1 & 0 & 0 & 0 & 0 & 0\\
0 & 1 & 0 & 0 & 0 & 0\\
0 & 0 & 1 & 0 & 0 & 0\end{pmatrix}, M_H=\begin{pmatrix}2/\lambda & 1 & 1 & 0 & 0\\
0 & 1/\lambda & 1 & 1 & 0\\
0 & 0 & 1/\lambda & 0 & 1\\
0 & 1 & 0 & 0 & 0\\
0 & 0 & 1 & 0 & 0\end{pmatrix}.$$

We can always remove one node in an isospectral reduction. Let us  remove node 4 from graph $G$. The weights of the edges after reduction become
$$R(i,j)=w(i,j)+w(i,4)\frac{w(4,j)}{\lambda},\quad i,j=1,2,3,5,6.$$
But $w(i,4)=0$ for all $i=2,3,5,6$, and $w(4,j)=0$ for $j=2,3,5,6$. The only weight that actually changes after the reduction is $R(1,1)=w(1,1)+w(1,4)w(4,1)/\lambda=2/\lambda$. All the other weights satisfy $R(i,j)=w(i,j),i\neq 1$ or $j\neq 1$. The reduced graph after removing node $4$ is identical to graph $H$.

Therefore $H$ is an isospectral reduction of $G$. The networks $H$ and $G$ will have the same reduction as long as we pick the same subset of vertices to reduce on.

We introduce now a few useful
notations. For any graph $G=(V,E,w)$, denote the maximum indegree by $m^-=\max\{d^-(v):v\in V\}$, the maximum outdegree by $m^+=\max\{d^+(v):v\in V\}$, and the maximum sum of indegree and outdegree as $m=\max\{d(v):v\in V\}$. We define a few different rules for picking a subset of the vertices of a graph.
$$\tau_1(G)=\{v\in V:d(v)>m/2\}; \tau_2(G)=\{v\in V:d^-(v)\ge m^-/2\}; \tau_3(G)=\{v\in V: d^-(v)>m^-/4\}.$$

$\tau_1$ picks the nodes whose sum of indegree and outdegree is greater than half of the maximum. $\tau_2$ picks the nodes whose indegree is greater than or equal to half of the maximum. And $\tau_3$ picks the nodes whose indegree is greater than a quarter of the maximum.

Now we apply these rules to $G$ and $H$ and see what happens. Consider the degrees of all the nodes in the two graphs. We list them in the following table
\ref{degree}.
\vspace{-8pt}
\begin{table}[h]
\centering
\caption{The degrees of each node in $G$ and $H$}\label{degree}
\begin{tabular}{|c|c|c|c|c|c|c|c|c|c|c|c|} 
\hline
graph & \multicolumn{6}{|c|}{$G$} & \multicolumn{5}{|c|}{$H$}\\
\hline
node & 1 & 2 & 3 & 4 & 5 & 6 & 1 & 2 & 3 & 5 & 6\\
\hline
indegree & 2 & 3 & 4 & 1 & 1 & 1 & 1 & 3 & 4 & 1 & 1\\
outdegree & 4 & 3 & 2 & 1 & 1 & 1 & 3 & 3 & 2 & 1 & 1\\
sum of indegree and outdegree & 6 & 6 & 6 & 2 & 2 & 2 & 4 & 6 & 6 & 2 & 2\\
\hline
\end{tabular}
\end{table}

Let us consider $\tau_1$ first. Both $G$ and $H$ have maximum sum of indegree and outdegree 6. $\tau_1(G)=\tau_1(H)=\{1,2,3\}$. $G$ and $H$ reduce to the same graph in one step under rule $\tau_1$, as shown in figure \ref{tau1}. So $G$ and $H$ are spectrally equivalent by rule $\tau_1$ with respect to both the 1-step definition in \cite{webb14} and the multi-step definition we have here. Also the reduced graph $A_1$ is an attractor for rule $\tau_1$. The 3 nodes have the same sum of indegree and outdegree, which is 4. To be more precise, if we write down the indegree, outdegree and the sum of the two, $(d^-,d^+,d)$ as an ordered triple for each node, all the nodes in $A_1$ are node 1 with $(1,3,4)$, node 2 with $(2,2,4)$ and node 3 with $(3,1,4)$, $d(1)=d(2)=d(3)=m(A_1)$.
\vspace{-3pt}
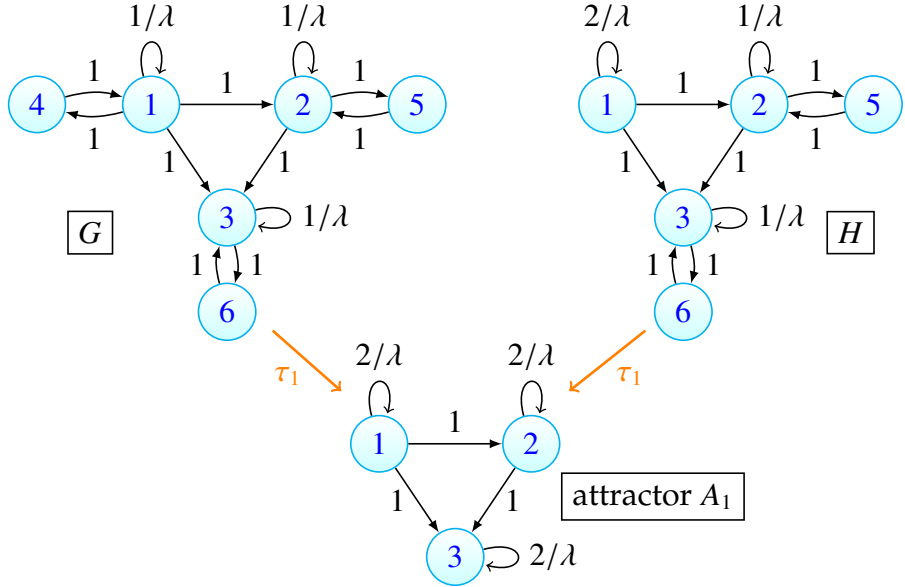
\begin {figure}[h]
\begin {tikzpicture}[-latex ,auto ,node distance =2 cm and 2cm ,on grid ,
semithick ,
state/.style ={ circle ,top color =white , bottom color = processblue!20 ,
draw,processblue , text=blue , minimum width =0.5 cm}]
\node[draw] at (-0.8,-1.7) {$G$};
\node[state] (A) at (0, 0) {$1$};
\node[state] (B) at (2,0) {$2$};
\node[state] (C) at (1,-1.5) {$3$};
\node[state] (D) at (-1.5,0) {$4$};
\node[state] (E) at (3.5,0) {$5$};
\node[state] (F) at (1,-2.75) {$6$};
\path (A) edge [] node [above] {$1$} (B);
\path (A) edge [] node [left] {$1$} (C);
\path (B) edge [] node [right] {$1$} (C);
\path (A) edge [bend left = 15] node [below] {$1$} (D);
\path (D) edge [bend left = 15] node [above] {$1$} (A);
\path (B) edge [bend left = 15] node [above] {$1$} (E);
\path (E) edge [bend left = 15] node [below] {$1$} (B);
\path (C) edge [bend left = 15] node [right] {$1$} (F);
\path (F) edge [bend left = 15] node [left] {$1$} (C);
\path (A) edge[loop above] node [] {$1/\lambda$} (A);
\path (B) edge [loop above] node [] {$1/\lambda$} (B);
\path (C) edge [loop right] node [] {$1/\lambda$} (C);

\node[draw] at (9.2,-1.7) {$H$};
\node[state] (G) at (6,0) {$1$};
\node[state] (H) at (8,0) {$2$};
\node[state] (I) at (7,-1.5) {$3$};
\node[state] (J) at (9.5,0) {$5$};
\node[state] (K) at (7,-2.75) {$6$};
\path (G) edge [] node [above] {$1$} (H)
(G) edge [] node [left] {$1$} (I)
(H) edge [] node [right] {$1$} (I)
(H) edge [bend left = 15] node [above] {$1$} (J)
(J) edge [bend left = 15] node [below] {$1$} (H)
(I) edge [bend left = 15] node [right] {$1$} (K)
(K) edge [bend left = 15] node [left] {$1$} (I)
(G) edge [loop above] node [] {$2/\lambda$} (G)
(H) edge [loop above] node [] {$1/\lambda$} (H)
(I) edge [loop right] node [] {$1/\lambda$} (I);
\draw[orange, ->, line width=1pt] (1.6,-3) to (2.5,-3.8) node at (1.8,-3.6) {$\tau_1$};
\draw[orange, ->, line width=1pt] (6.5,-3) to (5.5,-3.8) node at (6.3,-3.6) {$\tau_1$};

\node[state] (L) at (3,-4.5) {$1$};
\node[state] (M) at (5,-4.5) {$2$};
\node[state] (N) at (4,-6) {$3$};
\path (L) edge [] node [above] {$1$} (M)
(L) edge [] node [left] {$1$} (N)
(M) edge [] node [right] {$1$} (N)
(L) edge [loop above] node [] {$2/\lambda$} (L)
(M) edge [loop above] node [] {$2/\lambda$} (M)
(N) edge [loop right] node [] {$2/\lambda$} (N);
\node[draw] at (6.6,-5.2) {attractor $A_1$};
\end{tikzpicture}
\caption{Isospectral reductions by rule $\tau_1$}\label{tau1}
\end{figure}

Similarly, for rule $\tau_2$, we have $\tau_2(G)=\{1,2,3\}\neq\tau_2(H)=\{2,3\}$. Howeve, $\tau_2(\tau_2(G))=\{2,3\}=\tau(H)$. Under rule $\tau_2$, graph $G$ takes 2 reductions to reach the attractor $A_2$ while graph $H$ takes only one step (see figure \ref{tau2}). So $G$ and $H$ are spectrally equivalent with our generalized definition but not w.r.t. \cite{webb14}. In graph $A_2$, the degree triplet vectors for each node are node 2 with $(1,2,3)$ and node 3 with $(2,1,3)$. $d^-(2)=1=1/2m^-(A_2)=1/2d^-(3)$. One can see $A_1$ is an attractor by rule $\tau_1$ but not by rule $\tau_2$ since $d^-(1)=1<1/2d^-(3)=3/2$.

\vspace{-4pt}
\begin {figure}[h]
\begin {tikzpicture}[-latex ,auto ,node distance =2 cm and 2cm ,on grid ,
semithick ,
state/.style ={ circle ,top color =white , bottom color = processblue!20 ,
draw,processblue , text=blue , minimum width =0.5 cm}]
\node[draw] at (0,0) {$G$};

\node[draw] at (9,0) {$H$};

\node[state] (L) at (1,-3) {$1$};
\node[state] (M) at (3,-3) {$2$};
\node[state] (N) at (2,-4.5) {$3$};
\path (L) edge [] node [above] {$1$} (M)
(L) edge [] node [left] {$1$} (N)
(M) edge [] node [right] {$1$} (N)
(L) edge [loop above] node [] {$2/\lambda$} (L)
(M) edge [loop above] node [] {$2/\lambda$} (M)
(N) edge [loop right] node [] {$2/\lambda$} (N);
\node[draw] at (0.5,-4) {$A_1$};

\node[state] (A) at (6.5,-3.5) {$2$};
\node[state] (B) at (8.5,-3.5) {$3$};
\path (A) edge [] node [above] {$1$} (B)
(A) edge [loop above] node [] {$2/\lambda$} (A)
(B) edge [loop above] node [] {$2/\lambda$} (B);
\node[draw] at (7.5,-4.5) {attractor $A_2$};

\draw[orange, ->, line width=1pt] (0,-0.5) to (0.5,-2) node at (0,-1.5) {$\tau_2$};
\draw[orange, ->, line width=1pt] (9,-0.5) to (8.5,-2) node at (9,-1.5) {$\tau_2$};
\draw[orange, ->, line width=1pt] (4,-3.5) to (5.5,-3.5) node at (4.75,-3.25) {$\tau_2$};

\end{tikzpicture}
\caption{Isospectral reductions by rule $\tau_2$}\label{tau2}
\end{figure}
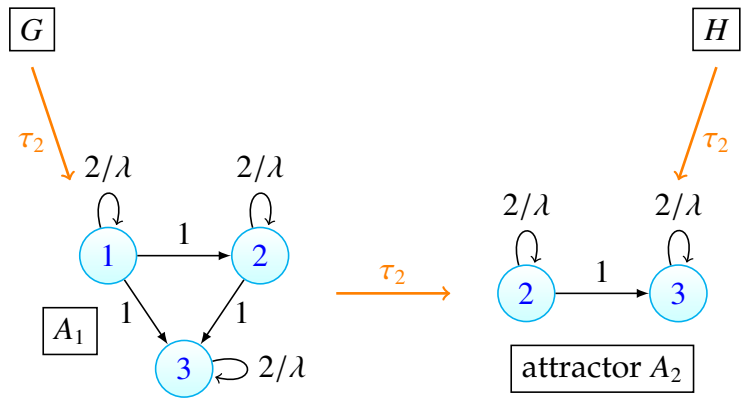

Lastly, for $\tau_3$,
$\tau_3(G)=\{1,2,3\}=\tau_3(\tau_3(G)),\tau_3(H)=\{2,3\}=\tau_3(\tau_3(H))$. $G$ and $H$ both reach an attractor in one step. But the attractors they reach are different. Under rule $\tau_3$, graph $G$ and $H$ are not isospectrally equivalent by either definition (see figure \ref{tau3}).

Here $A_1$ and $A_2$ are both attractors for rule $\tau_3$. For $A_1$, $d^-(1)=1,d^-(2)=2,d^-(3)=3$. For $A_2$ we have $d^-(2)=1,d^-(3)=2$. So $A_1$ is an attractor under rule $\tau_1$ and $\tau_3$ but not under $\tau_2$. $A_2$ is an attractor for all 3 rules we used in this example.

\begin {figure}[h]
\begin {tikzpicture}[-latex ,auto ,node distance =2 cm and 2cm ,on grid ,
semithick ,
state/.style ={ circle ,top color =white , bottom color = processblue!20 ,
draw,processblue , text=blue , minimum width =0.5 cm}]
\node[draw] at (0,0) {$G$};

\node[draw] at (9,0) {$H$};

\node[state] (L) at (1,-3) {$1$};
\node[state] (M) at (3,-3) {$2$};
\node[state] (N) at (2,-4.5) {$3$};
\path (L) edge [] node [above] {$1$} (M)
(L) edge [] node [left] {$1$} (N)
(M) edge [] node [right] {$1$} (N)
(L) edge [loop above] node [] {$2/\lambda$} (L)
(M) edge [loop above] node [] {$2/\lambda$} (M)
(N) edge [loop right] node [] {$2/\lambda$} (N);
\node[draw] at (0.5,-4) {$A_1$};

\node[state] (A) at (6.5,-3.5) {$2$};
\node[state] (B) at (8.5,-3.5) {$3$};
\path (A) edge [] node [above] {$1$} (B)
(A) edge [loop above] node [] {$2/\lambda$} (A)
(B) edge [loop above] node [] {$2/\lambda$} (B);
\node[draw] at (7.5,-4.5) {$A_2$};

\draw[orange, ->, line width=1pt] (0,-0.5) to (0.5,-2) node at (0,-1.5) {$\tau_3$};
\draw[orange, ->, line width=1pt] (9,-0.5) to (8.5,-2) node at (9,-1.5) {$\tau_3$};

\end{tikzpicture}
\caption{Isospectral reductions by rule $\tau_3$}\label{tau3}
\end{figure}
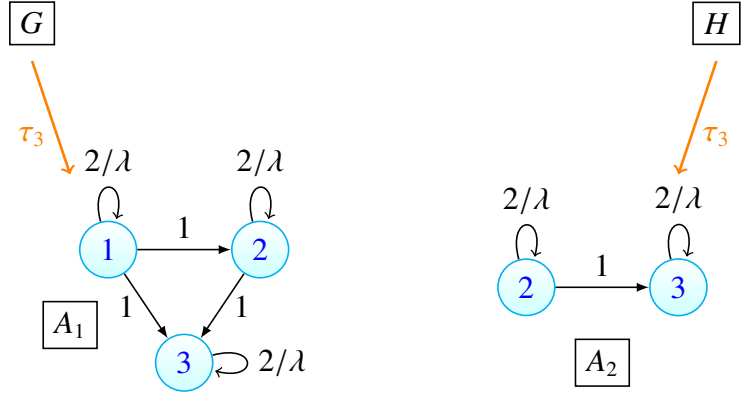

\begin{theorem}
Let $G=(V,E,w)$ with $w:E\to\C$. $S$ is a structural subset for graph $G$, then it remains to be a structural subset for any isospectral reduction (which contains $S$) of the graph $G$.
\end{theorem}
\begin{proof}
Suppose $S\subsetneq S'\subsetneq V$. Now we will show $S$ is also a structural subset for the reduced graph $R_{S'}(G)$.

(i) Any cycle (not a loop) in $R_{S'}(G)$ comes from a cycle in $G$, it has to contain a vertex in $S$.

(ii) For any $i\in S'\setminus S$, the new weight in $R_{S'}(G)$
$$\tilde w(i,i)=w(i,i)+\sum_{j\in V/S'}w(i,j)\frac{w(j,i)}{\lambda-w(j,j)}+\sum_{j\neq k,j,k\in V/S'}w(i,j)\frac{w(j,k)}{\lambda-w(j,j)}\frac{w(k,i)}{\lambda-w(k,k)}+\dots.$$

Since $w(i,i),w(j,j),w(k,k)\in\C$, the expression above shows $\tilde w(i,i)\neq\lambda$. We have $S$ is a structural subset of $R_{S'}(G)$.
\end{proof}

\begin{rk}
If we allow the original graph to take weights in $\W$, the above proof still holds as long as $\tilde w(i,i)\neq\lambda,\forall i\in S'\setminus S$. Since it's a zero measure set among all the possible values $\tilde w(i,i)$'s can take, we can say generally, the theorem is true for any graph with weights in $\W$ except for unusual cases.
\end{rk}
By the uniqueness of sequential graph reductions, we can see isospectral reduction is a dynamical system.

\section{Acknowledgements}

This work was partially supported by the NSF grant CCF-BSF 1615407 and the NIH grant 1RO1EBO25022-01.



\begin{thebibliography}{9}

\bibitem{newman06}
Mark Newman, Albert Laszlo Barabasi, Duncan J. Watts, The Structure and Dynamics of Networks, Princeton University Press 2006.

\bibitem{newman10}
Mark Newman, Networks: An Introduction, Oxford University Press, 2010.

\bibitem{webb14}
Leonid Bunimovich, Benjamin Webb, Isospectral Transformations, Springer, 2014.

\bibitem{webb17}
Leonid Bunimovich, D.C. Smith, Benjamin Webb, Specialization Models of Network Growth, arXiv 1712.01788(2017)

\bibitem{shu18}
Leonid Bunimovich, Lognmei Shu, Generalized Eigenvectors of Isospectral Transformations, Spectral Equivalence and Reconstruction of Original Networks, arXiv 1802.03410(2018).

\bibitem{ACM}
A. Clauset, C. Moore, M. Newman, Hierarchical structure and the prediction of missing links in networks, Nature 453(2008)98-101.

\bibitem{JKAM}
J. Leskovec, K. Lang, A. Dasgupta, M. Mahoney, Statistical Properties of Community Structure in Large Social and Information Networks, Proceedings of the 17th International Conference on World Wide Web (2008)695-704.

\end{thebibliography}

\end{document}